\documentclass[12pt]{article}
\usepackage{times}
\usepackage{amsmath}
\usepackage{latexsym}
\usepackage{amsfonts}
\usepackage{color}
\usepackage{amssymb}
\usepackage{latexsym}
\newcounter{theorem}

\renewcommand{\thetheorem}{\arabic{section}.\arabic{theorem}}

\newenvironment{thm}[1]{\par\addvspace{0.5cm}
    \begin{sloppypar}\refstepcounter{theorem}%
    {\bf #1 \thetheorem.}\it{}}{\end{sloppypar}}

\newcommand{\eh}{\hfill}\newlength{\sperr}

\newenvironment{theorem}{\begin{thm}{Theorem}} {\end{thm}}
\newenvironment{lemma}{\begin{thm}{Lemma}} {\end{thm}}
\newenvironment{corollary}{\begin{thm}{Corollary}} {\end{thm}}

\newenvironment{defi}[1]{\par\addvspace{0.5cm}
\begin{sloppypar}\refstepcounter{theorem}%
{\bf #1 \thetheorem.}\rm{}}{\end{sloppypar}}

\newenvironment{definition}{\begin{defi}{Definition}}{\end{defi}}
\newenvironment{example}{\begin{defi}{Example}}{\end{defi}}
\newenvironment{remark}{\begin{defi}{Remark}}{\end{defi}}

\newenvironment{proof}{{\settowidth{\sperr}{\rm Proof}
\par\addvspace{0.3cm}\parbox[t]{1.3\sperr}{\rm P\eh r\eh o\eh o\eh f\eh. }%
}}{\nopagebreak\mbox{}\hfill $\Box$\par\addvspace{0.25cm}}

\newcommand{\al}{\alpha}
\newcommand{\bt}{\beta}
\newcommand{\dl}{\delta}

\newcommand{\Om}{\Omega}
\newcommand{\lb}{\lambda}

\newcommand{\ve}{\varepsilon}
\newcommand{\gm}{\gamma}
\newcommand{\vi}{\varphi}
\newcommand{\Gm}{\Gamma}
\newcommand{\sg}{\sigma}

\newcommand{\intl}{\int\limits}

\newcommand{\rone}{\mathbb{R}^1}

\begin{document}

\centerline{\textbf{\Large Weighted Hardy and singular operators}} \centerline{\textbf{\Large in
Morrey spaces}}

\vspace{4mm} \centerline{\textbf{Natasha Samko}}

\begin{abstract}

We study the weighted boundedness of the Cauchy
 singular integral operator $S_\Gm$ in Morrey spaces $L^{p,\lambda}(\Gm)$ on curves satisfying
 the  arc-chord condition, for a class of  "radial type" almost monotonic weights.
 The non-weighted boundedness is shown to hold on an arbitrary Carleson curve. We show that the weighted
   boundedness is reduced to the boundedness
of weighted Hardy operators  in
 Morrey spaces $L^{p,\lambda}(0,\ell), \ell>0$. We find conditions for weighted Hardy
operators to be bounded  in Morrey spaces. To cover the case of curves we also extend the
boundedness of the Hardy-Littlewood maximal operator in Morrey spaces, known in the Euclidean
setting, to the case of Carleson curves.

\bigskip


\end{abstract}

\setcounter{equation}{0}\setcounter{theorem}{0}

\section{Introduction} \label{intro}

The well known Morrey spaces  $\mathcal{L}^{p,\lb}$ introduced in \cite{405a} in relation to the
study of partial differential equations, and  presented in various books, see \cite{187a},
\cite{348a}, \cite{655c}, were widely investigated during last decades, including the study of
classical operators of harmonic analysis - maximal, singular and potential operators - in these
spaces; we refer for instance to the  papers \cite{9e}, \cite{9f}, \cite{18f}, \cite{26c},
\cite{87b}, \cite{160bz}, \cite{107f},  \cite{412zz}, \cite{412z}, \cite{412zb}, \cite{463b},
\cite{472a}, \cite{473}, \cite{504za}, \cite{621a}, \cite{638a}, \cite{641a}, where  Morrey spaces
on metric measure spaces may be  also found. In particular, for the boundedness of the maximal
operator in Morrey spaces we refer to  \cite{87b}, while the boundedness of Calderon-Zygmund type
singular operators is known from \cite{472a}, \cite{473}, \cite{638a}.

Meanwhile weighted estimations of these operators in Morrey spaces
 were not studied (to our surprise, we did not find any
such weighted result for maximal and singular operators in the
literature). We are mainly interested in weighted estimations of
singular operators. In this paper we deal with the one-dimensional
case and study the weighted boundedness of the Cauchy singular
integral operator
\begin{equation}\label{4ss}
S_\Gm f(t)= \frac{1}{\pi}\int_{\Gm}\frac{f(\tau)\, d \tau}{\tau-t}
\end{equation}
along curves on complex plane, such a weighted estimation  being a key point for applications  to
the solvability theory of singular integral equations. We refer to \cite{171}, \cite{411z},
\cite{207}, \cite{208}, \cite{63} for this theory.

We obtain weighted estimates of the operator $S_\Gm$ in Morrey spaces  $\mathcal{L}^{p,\lb}(\Gm)$
along an arbitrary curve satisfying the arc-chord condition, in case of weights $\varrho$ of the
form
\begin{equation}\label{vydel}
\varrho(t)=\prod_{k=1}^N\vi_k(|t-t_k|), \quad t_k\in\Gm,
\end{equation}
with almost monotone functions $\vi_k$, the necessary and sufficient condition for the boundedness
being given in terms of the Matuszewska-Orlicz indices of these functions.

We start with the case where $\Gm=[0,\ell]$ is an interval of the real axis. We make use of the
known non-weighted boundedness of $S_\Gm$ in this case, which enables us to reduce the boundedness
of $S_\Gm$ with   weight \eqref{vydel} to the boundedness of weighted Hardy operators. We prove
their boundedness in the Morrey space $\mathcal{L}^{p,\lb}(0,\ell)$. Surprisingly we did not find
any statement on the boundedness of Hardy operators in Morrey spaces in the literature. In the
context of Morrey spaces, Hardy operators seem to have appeared only in \cite{69ab} in a different
aspect: the problem of the boundedness of the maximal operator in local Morrey spaces was reduced
to an $L_p$-boundedness of Hardy operators on a cone of monotone functions.

To cover the case of an arbitrary curve $\Gm$ we have to prove first the non-weighted boundedness
of the singular operator on curves, which was unknown, up to our knowledge. To this end, we first
prove the boundedness of the maximal operator along a Carleson curve, in
$\mathcal{L}^{p,\lb}(\Gm)$, but give the proof within the frameworks of homogeneous spaces with
constant dimension, Carleson curves being examples of such a space. Then we derive the non-weighted
boundedness of $S_\Gm$ via the Alvarez-P\'erez-type pointwise estimate
\begin{equation}\label{alvar}
M^\#(|S_\Gm f|^s)(t)\le C [Mf(t)]^s, \quad 0<s<1.
\end{equation}
known in the Euclidean setting for Calderon-Zygmund singular operators (\cite{19}) and extended to
the case of the operator $S_\Gm$ along Carleson curves in \cite{317b}, Proposition 6.2.

The paper is organized as follows. In Section \ref{prelim} we provide necessary preliminaries on
Morrey spaces on metric measure spaces, on almost monotonic  weights and their Matuszewska-Orlicz
indices. In Section \ref{Hardy} we give  sufficient conditions of the boundedness of weighted Hardy
operators in Morrey spaces $\mathcal{L}^{p,\lb}(0,\ell)$ in terms of the indices of the weight.
These conditions are necessary in the case of power weights. In Section \ref{weightedHilbert} we
give sufficient conditions of weighted boundedness of the singular operator along $(0,\ell)$, which
prove to be also necessary for power weights. In Section \ref{Singular} we first extend
Chiarenza-Frasca's proof (\cite{87b}) of the boundedness of the maximal operator  to the case of
metric measure spaces $X$ with constant dimension and prove the Fefferman-Stein inequality
$\|Mf\|_{L^{p,\lb}(X)}\le C \|M^\#f\|_{L^{p,\lb}(X)}$, to derive the non-weighted boundedness of
$S_\Gm$ via \eqref{alvar}. Finally in Section \ref{finalsection}, we prove the weighted boundedness
of $S_\Gm$.

\section{Preliminaries}\label{prelim}

\setcounter{equation}{0}\setcounter{theorem}{0}

\subsection{Morrey spaces on homogeneous spaces}\label{Morrey}

Let $(X,d,\mu)$ be  a homogeneous metric measure space with quasidistance $d$ and measure $\mu$. We
refer to \cite{97}, \cite{187}, \cite{225a} for analysis in homogeneous spaces. Morrey spaces on
metric measure spaces were studied in \cite{26c}, \cite{412za}, \cite{722b}. Our main goal in the
sequel is the case where $X$ is a Carleson curve in the complex plane with $\mu$ an arc-length,
although some auxiliary statements will be given in a more general setting.  By this reason we
restrict ourselves to the case where $X$ has constant dimension: there exists a number $N>0$ (not
necessarily integer) such that
\begin{equation}\label{2}
C_1r^N\le \mu B(x,r)\le C_2r^N,
\end{equation}
where the constants $C_1>0$ and $C_2>0$ do not depend on $x\in X$ and $r>0$. In this case the
Morrey space $\mathcal{L}^{p,\lb}(X)$ may be defined  by the norm:
\begin{equation}\label{1}
\|f\|_{p,\lb}=\sup\limits_{x\in X,r>0} \left\{\frac{1}{r^\lb}\intl_{B(x,r)}|f(y)|^p\,d\mu(y)
\right\}^\frac{1}{p},
\end{equation}
where $1\le p <\infty$ and  $0\le \lb < N$ and the standard notation $B(x,r)=\{y\in X: d(x,y):r\}$
is used.

\subsection{The case of Carleson curves}

Let $\Gm$ be a bounded rectifiable curve  on the complex plane $\mathbb{C}$. We denote
$\tau=t(\sg), \ \ \ \ t=t(s)$ where $\sg$ and $s$ stand for the arc abscissas of the points $\tau$
and $t$,
 and $d\mu(\tau)=d\sg$ will stand for the arc-measure on $\Gm$.
We also introduce the notation
 $$\Gm(t,r)=\{\tau\in\Gm:
|\tau-t|<r\}\quad \textrm{and} \quad \Gm_\ast(t,r)=\{\tau\in\Gm: |\sg-s|<r\},$$ so that
$\Gm_\ast(t,r)\subseteq \Gm(t,r),$ and denote $ \ell=\mu\Gm= \ \textrm{lengths of } \ \Gm.$

\begin{definition}\label{defC} A a curve $\Gm$ is said to be a \textit{Carleson curve}, if
$$\mu \Gm(t,r)\le C r$$
for all $t\in\Gm$ and $r>0$, where $C>0$ does not depend on $t$ and $r$. A curve $\Gm$ is said to
satisfy \textit{the arc-chord condition at a point} $t_0=t(s_0)\in\Gm$, if there exists a constant
$k>0$, not depending on $t$ such that
\begin{equation}\label{arcchord}
|s-s_0|\le k |t-t_0|, \quad  t=t(s)\in \Gm.
\end{equation}
Finally, a curve $\Gm$ is said to satisfy the \textit{(uniform) arc-chord condition}, if
\begin{equation}\label{arcchorduniform}
|s-\sg|\le k |t-\tau|, \quad  t=t(s), \tau=t(\sg)\in \Gm.
\end{equation}

\end{definition}
In the sequel  $\Gm$ is always assumed to be a Carleson curve.

The  Morrey spaces $\mathcal{L}^{p,\lb}(\Gm)$ on $\Gm$ are defined in the usual way, as in
\eqref{1},  via the norm
\begin{equation}\label{replace1}
\|f\|_{p,\lb}=\sup\limits_{t\in\Gm,r>0}
\left\{\frac{1}{r^\lb}\intl_{\Gm(t,r)}|f(\tau)|^p\,d\mu(\tau)
\right\}^\frac{1}{p},
\end{equation}
where $1\le p <\infty$ and  $0\le \lb \le 1$. For brevity we
denote  $\|f\|_p=\left(\intl_\Gm |f(\tau)|^p d\mu(\tau)
\right)^\frac{1}{p}$, so that
\begin{equation}\label{replace2}
\|f\|_{p,\lb}=\sup\limits_{t,r}
\left\|\frac{\chi_{\Gm(t,r)}(\cdot)}{r^\frac{\lb}{p}}f(\cdot)\right\|_p.
\end{equation}

\begin{remark}\label{rem6}
One may define another version $\mathcal{L}^{p,\lb}_\ast(\Gm)$ of the Morrey space, in terms of the
arc neigbourhood $\Gm_\ast(t,r)$ of the point $t\in\Gm$,  by the norm
\begin{equation}\label{replace1ast}
\|f\|_{p,\lb}^\ast=\sup\limits_{t\in\Gm,r>0}
\left\{\frac{1}{r^\lb}\intl_{\Gm_\ast(t,r)}|f(\tau)|^p\,d\mu(\tau) \right\}^\frac{1}{p}
\end{equation}
so that $\mathcal{L}^{p,\lb}_\ast(\Gm) \subseteq \mathcal{L}^{p,\lb}(\Gm)$ in case of an arbitrary
curve. These spaces coincide, up to equivalence of the norms, when $\Gm$ satisfies the arc-chord
condition. If $\Gm$ has casps, these spaces may be different. If, for instance,  a bounded curve
$\Gm$ satisfies the condition
$$C|s-\sg|^a\le |t-\tau|, \quad C>0$$
for some $a\ge 1$, then $\mathcal{L}^{p,a\lb}(\Gm)\subseteq \mathcal{L}^{p,\lb}_\ast(\Gm) \subseteq
\mathcal{L}^{p,\lb}(\Gm).$
\end{remark}

\begin{lemma}\label{replacelem0}
Let $\Gm$ be  bounded rectifiable curve. For the power function $|t-t_0|^\gm$, $t_0\in\Gm$, to
belong to the Morrey space $\mathcal{L}^{p,\lb}(\Gm), 1\le p<\infty, \ 0<\lb <1$, the condition
\begin{equation}\label{replace3a}
\gm \ge \frac{\lb-1}{p}
\end{equation}
is necessary. It is also sufficient if  $\Gm$ is a Carleson curve
\end{lemma}
\begin{proof}\
\textit{The  necessity  part}. Let $|t-t_0|^\gm\in \mathcal{L}^{p,\lb}(\Gm)$. We suppose that
$\gm<0$, since there is nothing to prove when $\gm\ge 0$.  With $t_0=t(s_0)$ we have
\begin{equation}\label{replace15}
\||\tau-t_0|^\gm\|\ge
\sup\limits_{r>0}\left(\frac{1}{r^\lb}\intl_{\Gm(t_0,r)}|\tau-t_0|^{\gm
p} d\mu(\tau) \right)^\frac{1}{p}
 \ge  \sup\limits_{r>0}\left(\frac{1}{r^\lb}\intl_{|\sg-s_0|<r}|\sg-s_0|^{\gm p}
d\sg \right)^\frac{1}{p},
\end{equation}
where we have taken into account that $\Gm(t,r)\supseteqq \Gm_\ast(t,r)$ and $\gm$ is negative.
Since $|t-t_0|^\gm\in L^{p}(\Gm)$, by similar arguments we see that  $\gm p>-1$. Then we get
$$\||\tau-t_0|^\gm\|_{p,\lb}\ge \left(\frac{2}{\gm p+1}\right)^\frac{1}{p} \sup\limits_{r>0}r^\frac{\gm
p+1-\lb}{p}
$$
which may be finite only when $\gm p+1-\lb\ge 0.$

\textit{The  sufficiency  part}.  Let $\gm \ge \frac{\lb-1}{p}$. Again we may assume that $\gm$ is
negative. To estimate
$$\||\tau-t_0|^\gm\|_{p,\lb}= \sup\limits_{t,r}\left(\frac{1}{r^\lb}\intl_{\Gm(t,r)}|\tau-t_0|^{\gm p}
d\mu(\tau) \right)^\frac{1}{p},
$$
we distinguish the cases $|t-t_0|>2r$ and $|t-t_0|\le 2r$. In the first case we have $|\tau-t_0|\ge
|t-t_0|-|\tau-t|>r$ so that $|\tau-t_0|^{\gm p}< r^{\gm}$ and then
$$\||\tau-t_0|^\gm\|_{p,\lb}\le \sup\limits_{t,r}\left(r^{\gm p-\lb}\intl_{\Gm(t,r)}
d\mu(\tau) \right)^\frac{1}{p}= Cr^\frac{\gm p-\lb-1}{p}<\infty$$ where we used the fact that $\Gm$
is a Carleson curve. In the case $|t-t_0|\le 2r$ we have $\Gm(t,r)\subseteq \Gm(t_0,3r)$. Then
$$\||\tau-t_0|^\gm\|_{p,\lb}\le \sup\limits_{r}\left(\frac{1}{r^\lb}\sum\limits_{k=0}^\infty
\intl_{\Gm_k(t_0,r)}|\tau-t_0|^{\gm p} d\mu(\tau) \right)^\frac{1}{p},
$$
where $ \Gm_k(t_0,r)=\{\tau: 3\cdot 2^{-k-1}r<|\tau-t_0|<3\cdot 2^{-k}r\}$. Hence
$$\||\tau-t_0|^\gm\|_{p,\lb}\le C \sup\limits_{r}\left(\frac{1}{r^{\lb-\gm p}}
\sum\limits_{k=0}^\infty \frac{1}{2^{\gm p k}}\intl_{\Gm(t_0,2^{-k+1}r)} d\mu(\tau)
\right)^\frac{1}{p}
$$
 and we arrive at
the conclusion by standard arguments.
\end{proof}
\begin{remark}\label{replacerem1} The case $\lb>0$ differs from the case $\lb=0$: when $\lb=0$, condition
\eqref{replace3a} must be replaced by the condition $\gm>-\frac{1}{p}.$
\end{remark}

In the limiting case $\mu=\frac{\lb-1}{p}$ admitted in Lemma \ref{replacelem0} for power functions,
it is not possible to take a power-logarithmic function, as shown in the next lemma.
\begin{lemma}\label{replacelem3}
Let $\Gm$ be  bounded rectifiable curve. The function
$|t-t_0|^\frac{\lb-1}{p}\ln^\nu\frac{A}{|t-t_0|}$, where $t_0\in \Gm, \nu >0$ and $A\ge D$, does
not belong to $\mathcal{L}^{p,\lb}(\Gm).$
\end{lemma}
\begin{proof}
As in  \eqref{replace15},  we have
\begin{equation}\label{replace15b}
\left\||\tau-t_0|^\frac{\lb-1}{p}\ln^\nu\frac{A}{|\tau-t_0|}\right\|_{p,\lb}\ge
\sup\limits_{r>0}\left(\frac{1}{r^\lb}\intl_{|\tau-t_0|<r}|\tau-t_0|^{\lb-1}\ln^{\nu
p}\frac{A}{|\tau-t_0|} d\mu(\tau) \right)^\frac{1}{p}
\end{equation}
$$\ge \sup\limits_{r>0}\left(\frac{1}{r^\lb}\intl_{|\sg-s_0|<r}|\sg-s_0|^{\lb-1}\ln^{\nu
p}\frac{A}{|\sg-s_0|} d\sg \right)^\frac{1}{p} $$
$$= \sup\limits_{r>0}\left(2\intl_0^1t^{\lb-1}\left(\ln\frac{A}{r}+\ln\frac{1}{t}\right)^{\nu p}
dt \right)^\frac{1}{p}\ge \sup\limits_{0<r<\dl}\left(2\ln^{\nu
p}\frac{A}{r}\intl_0^1 t^{\lb-1} dt \right)^\frac{1}{p}= \infty .
$$

\end{proof}
\begin{remark}\label{rem}
Statements similar to Lemmas \ref{replacelem0} and \ref{replacelem3} hold also for Morrey spaces
over bounded sets $\Om$ in $\mathbb{R}^n$: \\
1. \ \textit{ The power function $|x-x_0|^\gm$, where $x_0\in\Om$ belongs to the Morrey space
$\mathcal{L}^{p,\lb}(\Om), 1\le p<\infty, \  0<\lb <n$, if and only if $ \gm \ge \frac{\lb-n}{p}. $
In the case $x_0\in\partial \Om$, the condition $ \gm \ge \frac{\lb-n}{p}$ remains sufficient; it
is also necessary if the point  $x_0$ is a regular point of
the boundary in the sense that $|\{y\in\Om : |y-x_0|<r\}|\sim cr^n$.}\\
2. \  \textit{ The function $|x-x_0|^\frac{\lb-n}{p}\ln^\nu\frac{D}{|x-x_0|}, D>\textrm{diam}\,
\Om$, where $x_0\in \Om$ or $x_0$ is a regular point of $\partial\Om$, does not belong to
$\mathcal{L}^{p,\lb}(\Om).$ }
\end{remark}

\subsection{On admissible  weight functions}

In the sequel, when studying the singular operator $S_\Gm$ along a curve $\Gm$ in weighted Morrey
space, we will deal with weights  of the form
\begin{equation}\label{weight}
\varrho(t)=\prod\limits_{k=1}^N \vi_k(|t-t_k|), \quad t_k\in \Gm.
\end{equation}
We introduce below the class of weight functions $\vi_k(x), \ x\in [0,\ell],$ admitted for our
goals.

Although the functions $\vi_k$ should be defined only on $[0,d]$, where
$d=\textrm{diam}\Gm=\sup_{t,\tau\in\Gm}|t-\tau|< \ell$,  everywhere below we consider them as
defined on $[0,\ell]$.

\begin{definition}\label{def0}\\
1) \   By $W$ we denote the class of continuous and
positive functions  $\vi(x)$ on $(0,\ell]$,\\
2) \ by $W_0$ we denote the class of functions  $\vi\in W$ such that $\lim\limits_{x\to 0}\vi(x)=0$
and $\vi(x)$ is almost increasing;\\
 3) \ by $\widetilde{W}$ we denote the class of functions $\vi\in W$ such $x^\al\vi(x)\in W_0$ for some
 $\al =\al(\vi)>0$.
\end{definition}
\begin{definition}\label{def}
Let $x,y\in (0,\ell]$ and $  x_+ = \max(x,y),  \ x_- = \min(x,y).$
 By $\mathbf{V}_{\pm\pm} $ we denote the classes of functions $\vi\in W$ defined by the following
conditions
\begin{equation}\label{4}
\mathbf{V_{++}} : \hspace{37mm}  \left|\frac{\vi(x)-\vi(y)}{x-y}\right| \le C \frac{\vi(x_+)}{x_+}
, \hspace{50mm}
\end{equation}

\begin{equation}\label{5}
\mathbf{V_{--}} : \hspace{37mm}  \left|\frac{\vi(x)-\vi(y)}{x-y}\right| \le C \frac{\vi(x_-)}{x_-}
, \hspace{50mm}
\end{equation}
\begin{equation}\label{5a}
\mathbf{V_{+-}} : \hspace{37mm}  \left|\frac{\vi(x)-\vi(y)}{x-y}\right| \le C \frac{\vi(x_+)}{x_-}
, \hspace{50mm}
\end{equation}
\end{definition}
\begin{equation}\label{5b}
\mathbf{V_{-+}} : \hspace{37mm}  \left|\frac{\vi(x)-\vi(y)}{x-y}\right| \le C \frac{\vi(x_-)}{x_+}.
\hspace{50mm}
\end{equation}
Obviously, $\mathbf{V_{++}} \subset \mathbf{V_{+-}}$ and $\mathbf{V_{-+}} \subset \mathbf{V_{--}}$.

Let $0<y<x\le \ell$. It is easy to check that in the case of power function $\vi(x)=x^\al, \al
\in\rone$,
 we have
\begin{equation}\label{6}
\left|x^\al -y^\al\right| \le C (x-y)x^{\al-1} \quad \Longleftrightarrow \quad \al \ge 0
\end{equation}
\begin{equation}\label{7}
\left|x^\al -y^\al\right| \le C (x-y)y^{\al-1} \quad \Longleftrightarrow \quad \al \le 1,
\end{equation}
\begin{equation}\label{6a}
\left|x^\al -y^\al\right| \le C (x-y)\frac{x^\al}{y} \quad \Longleftrightarrow \quad \al \ge -1
\end{equation}
\begin{equation}\label{7a}
\left|x^\al -y^\al\right| \le C (x-y)\frac{y^\al}{x} \quad \Longleftrightarrow \quad \al \le 0,
\end{equation}
where the  constant  $C>0$ does not depend on $x$ and $y$.  Thus,
$$x^\al\in \mathbf{V_{++}}  \ \Longleftrightarrow \   \al\ge 0 , \ \quad \
x^\al\in \mathbf{V_{--}}  \ \Longleftrightarrow \  \ \quad \al\le 1 $$ and
$$x^\al\in \mathbf{V_{+-}}  \ \Longleftrightarrow \   \al\ge -1 , \ \quad \
x^\al\in \mathbf{V_{-+}}  \ \Longleftrightarrow \  \ \quad \al\le 0. $$

In the sequel we will mainly work with the classes $\mathbf{V_{++}}$ and $\mathbf{V_{-+}}$.

We also denote
$$W_1= \left\{\vi\in W : \ \frac{\vi(x)}{x} \ \textrm{is almost decreasing} \right\}.$$
\begin{remark}\label{rem}
Note that functions $\vi\in W_1$ satisfy the doubling condition $\vi(2x)\le C\vi(x).$
 For a function  $\vi\in W_1$,
condition \eqref{4} yields condition \eqref{5}, that is $ W_1\cap  \mathbf{V_{++}} \subseteq
W_1\cap \mathbf{V_{--}}.$
 The inverse embedding may be not true, as the above example of the power functions in \eqref{6} -\eqref{7a}
 shows.
\end{remark}

\vspace{2mm} In the following lemma we show that conditions \eqref{4} and \eqref{5} are fulfilled
automatically not only for power functions, but for an essentially larger class of functions (which
in particular may oscillate between two power functions with different exponents). Note that the
information about this class is given in terms of increasing or decreasing functions, without the
word "almost". Two statements  \textit{i)} and \textit{ii)} in Lemma \ref{lem} reflect in a sense
the modelling cases \eqref{6} and \eqref{7}.
\begin{lemma}\label{lem}
Let $\vi\in W$. Then\\
i) \ $\vi \in V_{++}$ in the case $\vi$ is increasing and the
function  $\frac{\vi(x)}{x^\nu}$ is decreasing  for some $\nu\ge 0$;\\
ii) \ $\vi \in V_{--}$ in the case $\frac{\vi(x)}{x}$ is  decreasing and  there exist a number
$\mu\ge 0$  such that $x^\mu\vi(x)$ is increasing;\\
iii) \ $\vi \in V_{-+}$ in the case $\vi(x)$ is  decreasing and  there exist a number $\mu\ge 0$
such that $x^\mu\vi(x)$ is increasing.
\end{lemma}
\begin{proof} \ The case \textit{i)}. Let $0<y<x\le \ell.$  Since $\frac{\vi(x)}{x^\nu}$ is decreasing, we have
$1-\cfrac{\vi(y)}{\vi(x)}\le 1- \cfrac{y^\nu}{x^\nu}$  or $$ \vi(x)-\vi(y) \le
\vi(x)\frac{x^\nu-y^\nu}{x^\nu}\le C \vi(x)\frac{x-y}{x}$$ by \eqref{6}. Since $\vi(x)-\vi(y)\ge 0$
in the case $\mu=0$, we arrive at \eqref{4}.

\ The cases \textit{ii)} and \textit{iii)}.  Taking again $y<x$, by the case \textit{i)} we have
that \eqref{4} holds for the function $x^\mu \vi(x)$, that is,
$$\left|\vi(x)-\frac{y^\mu}{x^\mu}\vi(y)\right|\le C \frac{\vi(x)}{x}(x-y).$$
Then
$$\left|\vi(x)-\vi(y)\right|\le C  \frac{\vi(x)}{x}(x-y)+ \frac{x^\mu-y^\mu}{x^\mu}\vi(y) $$
$$\le C  \frac{\vi(x)}{x}(x-y)+  C  \frac{x-y}{x}\vi(y)$$
by \eqref{6}. In the case  \textit{ii)} we use the fact that $\frac{\vi(x)}{x}\le \frac{\vi(y)}{y}$
and arrive at \eqref{5}. In the case  \textit{iii)} we have $\vi(x)\le \vi(y)$ and arrive at
 \eqref{5b}.
\end{proof}

In the case of differentiable functions $\vi(x),x>0,$ we arrive at the following sufficient
conditions for  $\vi$ to belong to the classes $\mathbf{V}_{++}, \mathbf{V}_{-+}$.
\begin{lemma}\label{lem1}
Let $\vi\in W \cap C^1((0,\ell])$. If there exist $\ve>0$ and $\nu\ge 0$  such that
$$ 0\le \frac{\vi^\prime(x)}{\vi(x)}\le \frac{\nu}{x} \quad \textrm{for}
 \quad  \ 0<x\le \ve, $$
then $\vi\in \mathbf{V}_{++}$. \ If there exist $\ve>0$  and $\mu \ge 0$  such that
$$
-\frac{\mu}{x}\le \frac{\vi^\prime(x)}{\vi(x)}\le 0 \  \quad \textrm{for} \quad 0<x\le \ve, $$ then
$\vi\in \mathbf{V}_{++}$.
\end{lemma}
\begin{proof}
Since $\min_{x\in [\ve,\ell]}\vi(x)>0$, by the differentiability of $\vi(x)$ beyond the origin,
inequalities \eqref{4}, \eqref{5} hold automatically when $x\ge \ve$, so the conditions of Lemma
\ref{lem}  should  be checked only on $(0,\ve]$. It suffices to  note that
$\frac{\vi^\prime(x)}{\vi(x)}\le \frac{\nu}{x} \Longleftrightarrow
\left[\frac{\vi(x)}{x^\nu}\right]^\prime\le 0$ and $\frac{\vi^\prime(x)}{\vi(x)}\ge -\frac{\mu}{x}
\Longleftrightarrow [x^\mu \vi(x)]^\prime\ge 0$.
\end{proof}

\begin{example}\label{ex}
Let $\al,\bt\in \rone$ and $A>\ell$. Then
$$x^\al\left(\ln \frac{A}{x}\right)^\bt \in \left\{\begin{array}{ll}
\mathbf{V}_{++}, & \textrm{if} \ \al >0, \bt\in \rone \ \textrm{or}\ \al=0 \ \textrm{and} \ \bt \le
0\\  \mathbf{V}_{-+}, & \textrm{if} \ \al <0, \bt\in \rone \ \textrm{or}\ \al=0 \ \textrm{and} \
\bt \ge 0. \end{array}\right.$$
\end{example}

\subsection{Matuszewska-Orlicz type indices}

It is known that the property of a function to be almost increasing or almost decreasing after the
multiplication (division) by a power function is closely related to the notion of the so called
Matuszewska-Orlicz indices. We refer to \cite{270a}, \cite{342}, \cite{382a} (p.20), \cite{382b},
\cite{387z}, \cite{539}, \cite{539d} for the properties of the indices of such a type. For a
function $\vi\in W_0$, the numbers
$$m(\vi)=\sup_{0<x<1}\frac{\ln\left(\limsup\limits_{h\to 0}
\frac{\vi(hx)}{\vi(h)}\right)}{\ln x}= \lim_{x\to 0}\frac{\ln\left(\limsup\limits_{h\to 0}
\frac{\vi(hx)}{\vi(h)}\right)}{\ln x}$$ and
$$M(\vi)=\sup_{x>1}\frac{\ln\left(\limsup\limits_{h\to 0}
\frac{\vi(hx)}{\vi(h)}\right)}{\ln x}=\lim_{x\to \infty}\frac{\ln\left(\limsup\limits_{h\to 0}
\frac{\vi(hx)}{\vi(h)}\right)}{\ln x}$$  are known as \textit{the Matuszewska-Orlicz type lower and
upper indices} of the function $\vi(r)$. Note that in this definition $\vi(x)$ need not to be an
$N$-function: only its behaviour at the origin is of importance.    Observe that
$$0\le m(\vi)\le
M(\vi) \le \infty \ \ \textrm{for}\ \  \vi\in W_0.$$ It is obvious that
\begin{equation}\label{na013}
m[x^\lb \vi(x)]=\lb + m(\vi),  \quad M[x^\lb \vi(x)]=\lb + M(\vi), \quad \quad \lb\in \mathbb{R}^1.
\end{equation}
 Consequently,
$$-\infty <m(\vi)\le M(\vi)\le \infty \quad \textrm{for} \quad \vi\in
\widetilde{W}.$$
\begin{definition}\label{def3}
We say that a function  $\vi\in W_0$ belongs to the Zygmund class $\mathbb{Z}^\bt$, $\bt\in
\mathbb{R}^1$, if
$$\int_0^h\frac{\vi(x)}{x^{1+\bt}}dx \le c\frac{\vi(h)}{h^\bt},$$
and to the Zygmund class $\mathbb{Z}_\gm$, $\gm\in \mathbb{R}^1$, if
$$\int_h^\ell\frac{\vi(x)}{x^{1+\gm}}dx \le c\frac{\vi(h)}{h^\gm}.$$
We also denote
$$\mathbb{Z}^\bt_\gm : =\mathbb{Z}^\bt\bigcap \mathbb{Z}_\gm ,$$
the latter class being also known as Bary-Stechkin-Zygmund class \cite{46}.
\end{definition}
 The following statement is
known, see \cite{270a}, Theorems 3.1 and 3.2.

\begin{theorem}\label{tstasda3zl} Let $\vi\in \widetilde{W}$ and $\bt,\gm \in \mathbb{R}^1$.
Then
$$\vi \in \mathbb{Z}^\bt  \Longleftrightarrow  m(\vi)>\bt \quad \textrm{and} \quad
 \vi \in \mathbb{Z}_\gm   \Longleftrightarrow  M(\vi)<\gm.
 $$
Besides this
\begin{equation}\label{na01}
m(\vi) =\sup\left\{\dl> 0: \  \ \frac{\vi(x)}{x^{\dl}} \ \ \ \ \textrm{is almost increasing}
\right\},
\end{equation}
and
\begin{equation}\label{na0}
M(\vi) =\inf\left\{\lb> 0: \  \ \frac{\vi(x)}{x^{\lb}} \ \ \ \ \textrm{is almost decreasing}
\right\}.
\end{equation}
\end{theorem}
\begin{remark}\label{rem2}
Theorem \ref{tstasda3zl} was formulated in \cite{270a} for $\bt\ge 0,\gm> 0$  and $\vi\in W_0$. It
is evidently true when these exponents are negative, and for $\vi\in \widetilde{W}$, the latter
being obvious by the definition of the class $\widetilde{W}$ and formulas \eqref{na013}.
\end{remark}

\section{Weighted Hardy operators in Morrey spaces}\label{Hardy}

\setcounter{equation}{0}\setcounter{theorem}{0}

\subsection{The case of power weights}\label{power}

Let
\begin{equation}\label{confus}
H_\bt f(x)= x^{\bt-1}\intl_0^x \frac{f(t)dt}{t^\bt},  \quad
 \mathcal{H}_\bt f(x) = x^\bt\intl_x^\ell \frac{f(t)dt}{t^{\bt+1}}.
 \end{equation}

\begin{theorem}\label{th1}
Let $0<\ell\le \infty$. The operators $H_\bt$ and $\mathcal{H}_\bt$ are bounded in the Morrey space
$\mathcal{L}^{p,\lb}([0,\ell])$, $\ 1\le p<\infty, 0\le \lb <1$, if and only if
\begin{equation}\label{13}
\bt< \frac{\lb}{p}+\frac{1}{p^\prime} \quad \textrm{and} \quad \bt > \frac{\lb}{p}-\frac{1}{p},
\end{equation}
respectively.
\end{theorem}

\begin{proof}\\
\textit{"If" part}.
We may assume that $f(x)\ge 0$. First we observe that
$$H_\bt f(x)= \intl_0^1 \frac{f(xt)dt}{t^\bt}  \quad \textrm{and} \quad
 \mathcal{H}_\bt f(x) = \intl_1^\frac{\ell}{x} \frac{f(xt)dt}{t^{\bt+1}} \le \intl_1^\infty
 \frac{f(xt)dt}{t^{\bt+1}}$$
under the assumption that $f(x)$ is continued as $f(x)\equiv 0$ for $x>\ell$ in the inequality for
$\mathcal{H}_\bt f(x)$. For $H_\bt f$ we have
 $$  \left\|H_\bt f\right\|_{p,\lb}=
\sup\limits_{x,r}\left\|\frac{\chi_{B(x,r)}(y)}{r^\lb} H_\bt f(y)\right\|_p= \sup\limits_{x,r}
\left\{\frac{1}{r^\lb}\intl_0^\infty \left|\chi_{B(x,r)}(y)\intl_0^1\frac{f(yt)}{t^\bt}dt\right|^p
dy\right\}^\frac{1}{p}.$$ Then by Minkowsky inequality we obtain
$$ \left\|H_\bt f\right\|_{p,\lb} \le \sup\limits_{x,r}\intl_0^1 \frac{dt}{t^\bt}
\left\{\intl_0^\infty \left|\frac{\chi_{B(x,r)}(y)}{r^\lb} f(yt)\right|^p dy\right\}^\frac{1}{p}.$$
Hence, by the change of variables we get
$$ \left\|H_\bt f\right\|_{p,\lb} \le \sup\limits_{x,r}\intl_0^1 \frac{dt}{t^{\bt+\frac{1}{p}}}
\left\{\intl_0^\infty \left|\frac{\chi_{B(x,r)}\left(\frac{y}{t}\right)}{r^\frac{\lb}{p}}
f(y)\right|^p dy\right\}^\frac{1}{p}.$$ It is easy to see that
$$\chi_{B(x,r)}\left(\frac{y}{t}\right) = \chi_{B(tx,tr)}(y).$$
Therefore,
$$ \left\|H_\bt f\right\|_{p,\lb} \le
\sup\limits_{x,r}\intl_0^1 \frac{dt}{t^{\bt+\frac{1}{p}}} \left\{\intl_0^\infty
\left|\frac{\chi_{B(tx,tr)}(y)}{r^\frac{\lb}{p}} f(y)\right|^p dy\right\}^\frac{1}{p}$$
$$=\sup\limits_{x,r}\intl_0^1 \frac{dt}{t^{\bt+\frac{1-\lb}{p}}} \left\{\intl_0^\infty
\left|\frac{\chi_{B(tx,tr)}(y)}{(tr)^\frac{\lb}{p}} f(y)\right|^p dy\right\}^\frac{1}{p} $$ $$  \le
\intl_0^1 \frac{dt}{t^{\bt+\frac{1-\lb}{p}}} \sup\limits_{x,r}\left\{\intl_0^\infty
\left|\frac{\chi_{B(x,r)}(y)}{r^\lb} f(y)\right|^p dy\right\}^\frac{1}{p} =\frac{1}{\frac{\lb}{p}
+\frac{1}{p^\prime}-\bt}\left\|f\right\|_{p,\lb}$$

\vspace{2mm}

Similarly for the operator $\mathcal{H}_\bt$  we obtain
 $$  \left\|\mathcal{H}_\bt f\right\|_{p,\lb}=
\sup\limits_{x,r}\left\|\frac{\chi_{B(x,r)}(y)}{r^\frac{\lb}{p}} \mathcal{H}_\bt f(y)\right\|_p \le
\sup\limits_{x,r}\intl_1^\infty \frac{dt}{t^{1+\bt}} \left\{\intl_0^\infty
\left|\frac{\chi_{B(x,r)}(y)}{r^\frac{\lb}{p}} f(yt)\right|^p dy\right\}^\frac{1}{p}$$
$$ =  \sup\limits_{x,r}\intl_1^\infty \frac{dt}{t^{1+\bt+\frac{1}{p}}}
\left\{\intl_0^\infty \left|\frac{\chi_{B(x,r)}\left(\frac{y}{t}\right)}{r^\frac{\lb}{p}}
f(y)\right|^p dy\right\}^\frac{1}{p}$$
$$  =
\sup\limits_{x,r}\intl_1^\infty \frac{dt}{t^{1+\bt+\frac{1}{p}}} \left\{\intl_1^\infty
\left|\frac{\chi_{B(tx,tr)}(y)}{r^\frac{\lb}{p}} f(y)\right|^p dy\right\}^\frac{1}{p}
=\sup\limits_{x,r}\intl_1^\infty \frac{dt}{t^{1+\bt+\frac{1}{p}-\frac{\lb}{p}}}
\left\{\intl_0^\infty \left|\frac{\chi_{B(tx,tr)}(y)}{(tr)^\frac{\lb}{p}} f(y)\right|^p
dy\right\}^\frac{1}{p} $$ $$  \le \intl_1^\infty \frac{dt}{t^{1+\bt+\frac{1-\lb}{p}}}
\sup\limits_{x,r}\left\{\intl_0^\infty \left|\frac{\chi_{B(x,r)}(y)}{r^\frac{\lb}{p}} f(y)\right|^p
dy\right\}^\frac{1}{p} =\frac{1}{\bt+\frac{1-\lb}{p}}\left\|f\right\|_{p,\lb}$$

\vspace{2mm} \textit{"Only if" part}. The necessity of the condition $\bt<
\frac{\lb}{p}+\frac{1}{p^\prime}$  for the operator $H_\bt$ is well known in the case $\lb=0$. Let
$\lb>0$. It suffices to observe that the function $f(x)=x^\frac{\lb-1}{p}$ belongs to
$\mathcal{L}^{p,\lb}$ by Lemma \ref{replacelem3}, but the operator $H_\bt$ on this function exists
only when $\bt< \frac{\lb}{p}+\frac{1}{p^\prime}$.

With the same example $f(x)=x^\frac{\lb-1}{p}$ the case of the  operator $\mathcal{H}_\bt$ may be
similarly considered.

\end{proof}

\begin{corollary}\label{cor2}
$\|H_\bt\|_{\mathcal{L}^{p,\lb}}\le \frac{1}{\frac{\lb}{p} +\frac{1}{p^\prime}-\bt}, $ \
$\|H_\bt\|_{\mathcal{L}^{p,\lb}}\le  \frac{1}{\bt+ \frac{1-\lb}{p}}.$
\end{corollary}

\vspace{3mm} Observe that Theorem \ref{th1} is a particular case of the following statement for the
operators of the form
$$Af(x)= \intl_0^\infty a(t)f(xt)dt= \frac{1}{x}\intl_0^\infty a\left(\frac{t}{x}\right)f(t)dt$$
with homogeneous kernel. Under the choice $a(t)=\chi_{[0,1]}(t)$ we obtain the operator $H_\bt$,
while  taking $a(t)=\chi_{[1,\infty]}(t)$ we get at the majorant of the operator $\mathcal{H}_\bt$.

\begin{theorem}\label{th2}
Let $1\le p <\infty, \ 0\le \lb <1$ and  $C= \intl_0^\infty t^{\frac{\lb-1}{p}}|a(t)|dt <\infty .$
Then
\begin{equation}\label{14}
\|Af\|_{\mathcal{L}^{p,\lb}(\mathcal{R}_+^1)} \le C \|f\|_{\mathcal{L}^{p,\lb}(\mathcal{R}_+^1)}.
\end{equation}
\end{theorem}

The proof follows the same lines as for Theorem \ref{th1}.

\subsection{The case of general weights}\label{general}

\begin{theorem}\label{theor}
Let $\vi\in \widetilde{W}\cap \left(V_{++}\cup V_{-+}\right)$. Then the weighted  Hardy operators
$H_\vi$ and $\mathcal{H}_\vi$ are bounded in the Morrey spaces $\mathcal{L}^{p,\lb}([0,\ell]), \
1\le p<\infty, \ 0\le \lb <1,  0<\ell < \infty$,  if
\begin{equation}\label{usedli}
\vi \in \mathbb{Z}_{\frac{\lb}{p}+\frac{1}{p^\prime}} \quad \textrm{and} \quad \vi \in
\mathbb{Z}^\frac{\lb-1}{p},
\end{equation}
respectively, or equivalently,
\begin{equation}\label{used}
M(\vi)<\frac{\lb}{p}+\frac{1}{p^\prime} \quad \quad \textrm{for the operator} \quad H_\vi
\end{equation}
and
$$m(\vi) >\frac{\lb}{p}-\frac{1}{p} \quad \quad \textrm{for the operator} \quad \mathcal{H}_\vi.$$
The conditions
\begin{equation}\label{necessH}
m(\vi)\le \frac{\lb}{p}+\frac{1}{p^\prime}, \quad M(\vi) \ge\frac{\lb}{p}-\frac{1}{p}
\end{equation}
 are necessary for the boundedness of the operators $H_\bt$ and $\mathcal{H}_\bt$,
respectively.
\end{theorem}

\begin{proof}
By \eqref{na01} and \eqref{na0}, the function $\cfrac{\vi(x)}{x^{m(\vi)-\ve}}$ is almost
increasing, while  $\cfrac{\vi(x)}{x^{M(\vi)+\ve}}$ is almost decreasing for any $\ve>0$.
Consequently,
$$C_1 \frac{x^{m(\vi)-\ve}}{t^{m(\vi)-\ve}}\le \frac{\vi(x)}{\vi(t)}\le C_2 \frac{x^{M(\vi)+\ve}}{t^{M(\vi)+\ve}}$$
and then
\begin{equation}\label{oha}
C_1 x^{m(\vi)-\ve-1}\intl_0^x \frac{f(t)\, dt}{t^{m(\vi)- \ve}} \le H_\vi f(x)\le C_2
x^{M(\vi)+\ve-1}\intl_0^x \frac{f(t)\, dt}{t^{M(\vi)+\ve}}
\end{equation}
supposing that $f(t) \ge
0$. Therefore, the operator $H_\vi$ is bounded by Theorem \ref{th1} for the  Hardy operators with
power weights, if $M(\vi)+\ve< \frac{\lb}{p}+\frac{1}{p^\prime}$, which
 is satisfied under the choice of $\ve>0$ sufficiently small, the latter being possible by
\eqref{used}. It remains to recall that condition \eqref{used}  is equivalent to the assumption
$\vi \in \mathbb{Z}_{\frac{\lb}{p}+\frac{1}{p^\prime}}$ by Theorem \ref{tstasda3zl}. The necessity
of the condition $m(\vi)\le \frac{\lb}{p}+\frac{1}{p^\prime}$ follows from the left-hand side
inequality in \eqref{oha}.

Similarly one may treat the case of the operator $\mathcal{H}_\vi$.
\end{proof}

\section{Weighted boundedness of the Hilbert transform in Morrey spaces}\label{weightedHilbert}
\setcounter{equation}{0}\setcounter{theorem}{0}

We start with the Cauchy singular integral along the real line or an interval ($\Gm=\rone$ or
$\Gm=[0,\ell]$) and denote
\begin{equation}\label{Hilbert}
S f(x)= \frac{1}{\pi}\int_{\rone}\frac{f(t)\, dt}{t-x}, \quad x\in \rone ; \quad
\mathbb{H}f(x)=\frac{1}{\pi}\intl_0^\ell \frac{f(t)\,dt}{t-x}, \quad 0<x<\ell\le \infty .
\end{equation}
In \cite{472a} there was proved the boundedness of  a class of Calderon-Zygmund  operators, which
includes in particular the following statement.
\begin{theorem}\label{th}
The operator  $S$ is bounded in the space $\mathcal{L}^{p,\lb}(\rone)$, $\ 1<p<\infty, 0\le \lb
<1.$
\end{theorem}

\begin{corollary}\label{cor}
The Hilbert transform operator $\mathbb{H}$ is bounded in the space
$\mathcal{L}^{p,\lb}([0,\ell])$, $\ 1<p<\infty, 0\le \lb <1.$
\end{corollary}

\subsection{Reduction of the Hilbert transform  operator with weight to the Hardy operators}

  The boundedness of the singular operator
$\mathbb{H}$ in the space $\mathcal{L}^{p,\lb}([0,\ell],\varrho)$ with a weight $\varrho$ is  the
same as the boundedness of the operator $\varrho \mathbb{H}\frac{1}{\varrho}$ in the space
$\mathcal{L}^{p,\lb}([0,\ell])$. In view of Corollary \ref{cor}, the latter boundedness will follow
from  the boundedness of the operator
\begin{equation}\label{7b}
K f(x): \ = \ \left(\varrho \mathbb{H}\frac{1}{\varrho}-\mathbb{H} \right) f(x) = \intl_0^\ell
K(x,t)f(t)\, dt,
\end{equation}
 where
$$K(x,t): = \frac{\varrho(x)-\varrho(t)}{\varrho(t)(t-x)}=
\frac{\vi(|x-x_0|)-\vi(|t-x_0|)}{\vi(|t-x_0|)(t-x)}$$ in the case $\varrho(x)=\vi(|x-x_0|), x_0\in
[0,\ell]$.

\subsubsection{The case $x_0=0$}
We start with the case $x_0=0$, so that $K(x,t) = \frac{\vi(x)-\vi(t)}{\vi(t)(t-x)}.$

\begin{lemma}\label{lem2}
The kernel $K(x,t)$ admits the estimate
\begin{equation}\label{8}
|K(x,t)|\le \left\{\begin{array}{ll}\cfrac{C}{x}\cfrac{\vi(x)}{\vi(t)},  &\  \textrm{if} \ \ t<x\\
\cfrac{C}{t}, &\  \textrm{if} \ \ t>x
\end{array} \right.
\end{equation}
when $\vi\in \mathbf{V}_{++}$, and
\begin{equation}\label{9}
|K(x,t)|\le \left\{\begin{array}{ll}\cfrac{C}{x},  &\  \textrm{if} \ \ t<x\\
\cfrac{C}{t}\cfrac{\vi(x)}{\vi(t)}, &\  \textrm{if} \ \ t>x
\end{array} \right.
\end{equation}
when $\vi\in \mathbf{V}_{-+}$.
\end{lemma}
\begin{proof}
Estimates \eqref{8}-\eqref{9} follow immediately from the definition of the classes
$\mathbf{V}_{++}, \mathbf{V}_{-+}$.
\end{proof}

\begin{corollary}\label{cor1}
The operator $K=\varrho \mathbb{H}\frac{1}{\varrho}-\mathbb{H} $ is dominated by the weighted Hardy
operators
\begin{equation}\label{10}
|K f(x)| \le C \frac{\vi(x)}{x}\intl_0^x \frac{|f(t)|dt}{\vi(t)} + C \intl_x^\ell
\frac{|f(t)|dt}{t}
\end{equation}
when $\vi \in \mathbb{V}_{++}$, and
\begin{equation}\label{10sw}
|K f(x)| \le \frac{C}{x}\intl_0^x |f(t)|dt + C \vi(x)\intl_x^\ell \frac{|f(t)|dt}{t\vi(t)}
\end{equation}
when $\vi \in \mathbb{V}_{-+}$. In particular, when $\vi(x)=x^\al,$
\begin{equation}\label{11}
|K f(x)| \le \frac{C}{x}\intl_0^x \left(\frac{x}{t}\right)^{\max(\al,0)}|f(t)|dt + C \intl_x^\ell
\left(\frac{x}{t}\right)^{\min(\al,0)} \frac{|f(t)|dt}{t}.
\end{equation}
\end{corollary}

In the sequel we use the notation
\begin{equation}\label{12}
H_\vi f(x)= \frac{\vi(x)}{x}\intl_0^x \frac{f(t)dt}{\vi(t)},  \quad
 \mathcal{H}_\vi f(x) = \vi(x)\intl_x^\ell \frac{f(t)dt}{t\vi(t)}
\end{equation}
  without fear of confusion with notation in \eqref{confus}

\begin{corollary}\label{cor5}
Let $\vi\in V_{++}\cup V_{-+}$. By \eqref{10}-\eqref{10sw}, the boundedness of the Hardy operators
$H_\vi$ and
 $\mathcal{H}_\vi$ in Morrey space $\mathcal{L}^{p,\lb}(0,\ell)$ yields that of the weighted singular  operator
 $\varrho\mathbb{H}\frac{1}{\varrho}, \ \varrho(x)=\vi(x), \ 1<p<\infty, \ 0\le \lb<1$.
 \end{corollary}

\subsubsection{The case $x_0\ne 0$}

 The
following  simple technical fact is valid.

\begin{lemma}\label{techn}
Let $-\infty\le a<b\le \infty$, $x_0\in(a,b)$ and let $\vi(x)$ be a non-negative function on
$\mathbb{R}^1_+$. If the operator
$$K f(x) =\intl_{x_0}^b\left|\frac{\vi(t-x_0)-\vi(x-x_0)}{t-x}\right|
\frac{f(t)}{\vi(t-x_0)}dt, \quad x_0<x<b$$ is bounded in the space $L^{p,\lb}([x_0,b]),$ then the
operator
$$\widetilde{K} f(x) =\intl_{a}^b\left|\frac{\vi(|t-x_0|)-\vi(|x-x_0|)}{t-x}\right|
\frac{f(t)}{\vi(|t-x_0|)}dt, \quad a<x<b$$ is bounded in the space $L^{p,\lb}([a,b]).$
\end{lemma}
\begin{proof}
Without loss of generality we may take $-a=b>0$ and $x_0=0$. Splitting the square $Q=\{(x,t):
-a<x<a,-a<t<a\}$ into the sum of 4 squares $Q=Q_{++}+Q_{--}+Q_{-+}+Q_{+-}$, where the first sign in
the index corresponds to the sign of $x$ and the second to that of $t$, we reduce the boundedness
of the operator $\widetilde{K}$ to that of the corresponding operators $\widetilde{K}_{++},
\widetilde{K}_{--}, \widetilde{K}_{-+}, \widetilde{K}_{+-}$. The operators $\widetilde{K}_{++}$ and
$\widetilde{K}_{--}$ are bounded, the former by assumption, the latter being obviously reduced to
the former. Because of  the evenness of the function $\vi(t)$, the kernels of the operators
$\widetilde{K}_{-+}$ and $\widetilde{K}_{+-}$ are
 obviously dominated by the kernels  of the operator $\widetilde{K}_{++}$, which completes the
proof.
\end{proof}

By Lemma \ref{techn}, the validity of the statement of Corollary \ref{cor5} in the case $x_0\ne$
follows from the case $x_0=0.$

\subsection{Weighted  boundedness of the Hilbert transform operator; the case of power weights}

\begin{theorem}\label{cor3} The weighted singular operator
$$ S_\al f(x)= \frac{x^\al}{\pi}\intl_0^\ell\frac{f(t)\,dt}{t^\al (t-x)} $$ is bounded in the
space $\mathcal{L}^{p,\lb}([0,\ell]),$ where $0<\ell \le\infty$, $1<p<\infty, \ 0\le \lb <1$, if
and only if
\begin{equation}\label{16}
-\frac{1}{p} <\al-\frac{\lb}{p}< \frac{1}{p^\prime}.
\end{equation}
\end{theorem}
\begin{proof} \textit{The "if" part}. \
The case $\lb=0$ is well known (Babenko weighted theorem, see for instance, \cite{207}, p.30). Let
$\lb\ne 0$.    By Corollary \ref{cor}, the boundedness of $S_\al$ is equivalent to that of the
difference
$$Kf (x):=(S_\al-S)f(x)= \frac{1}{\pi}\intl_0^\ell \frac{x^\al-t^\al}{t^\al(t-x)}f(t)\,dt.$$
By Corollary \ref{cor5}, it suffices  to have the boundedness of the Hardy operators $H_{\bt_1}$
with $ \bt_1=\max(\al,0)$ and $\mathcal{H}_{\bt_2}$ with $ \bt_2=\min(\al,0)$. Applying Theorem
\ref{th1}, we obtain that inequalities \eqref{16} are sufficient for the boundedness of the
operator $K$.

\vspace{2mm}\textit{The "only if" part}. \ It suffices to consider the case $\ell <\infty$.

The necessity of condition \eqref{16} in the case $\lb=0$ is well
known, see for instance \cite{63}, Lemma 4.6. Let $0<\lb<1$.
Suppose that $\al\le \frac{\lb-1}{p}$. To show that the operator
$S_\al$ is not bounded, we choose $f(t)=t^\frac{1-\lb}{p}$, which
is in $\mathcal{L}^{p,\lb}([0,\ell])$ by Lemma \ref{replacelem0}.
Then in the case $\al <\frac{\lb-1}{p}$ we have
\begin{equation}\label{asympt}
S_\al f(x)= \frac{x^\al}{\pi} \intl_0^\ell
\frac{t^{\frac{\lb-1}{p}-\al}}{t-x}dt \sim cx^\al \quad
\textrm{as} \quad x\to 0
\end{equation}
 with $c=
\frac{\ell^{\frac{\lb-1}{p}-\al}}{\frac{\lb-1}{p}-\al}.$ Since
$\al<\frac{\lb-1}{p}$, the function $S_\al f(x)\sim cx^\al$ proves
to be not in $\mathcal{L}^{p,\lb}([0,\ell])$. In the remaining
case $\al=\frac{\lb-1}{p}$, the singular integral
\begin{equation}\label{logarifm}
 \intl_0^\ell \frac{dt}{t-x}\sim \ln \frac{1}{x} \quad \textrm{as}
 \quad  x\to 0
 \end{equation}
 has
a logarithmic singularity and then the function $S_\al f(x)\sim
cx^\al \ln\frac{2\ell}{x}$ proves to be not in
$\mathcal{L}^{p,\lb}([0,\ell])$ by Lemma \ref{replacelem3}.

Finally, the necessity of the condition
$\al<\frac{\lb}{p}+\frac{1}{p^\prime}$ follows from the simple
fact that in the case $\al\ge \frac{\lb}{p}+\frac{1}{p^\prime}$
the weighted singular integral  $S_\al f$ does not exist on all
the functions $f\in \mathcal{L}^{p,\lb}([0,\ell])$. Indeed,  take
$f(t)=t^\frac{\lb-1}{p}\in \mathcal{L}^{p,\lb}([0,\ell])$, then
$\frac{f(t)}{t^\al}=\frac{1}{t^{\al+\frac{1-\lb}{p}}}$ with
$\al+\frac{1-\lb}{p}\ge 1$ is not in $L^1([0,\ell])$, while
belonging of a function to  $L^1$ is  a necessary condition for
the almost everywhere existence of the singular integral.
\end{proof}

\begin{corollary}\label{cor4}
Let $-\infty \le a<b\le \infty$ and $$\varrho(x)=\prod_{k=1}^N|x-x_k|^{\al_k}$$ where $x_k$ are
arbitrary finite points in $[a,b]$. The singular operator $S$ is bounded in the space
$L^{p,\lb}([a,b]\varrho)$, if and only if

\begin{equation}\label{16k}
\frac{\lb}{p}-\frac{1}{p} <\al_k<\frac{\lb}{p}+ \frac{1}{p^\prime}, \quad k=1,2,..., N.
\end{equation}
\end{corollary}
\begin{proof}
The case of a singe point $x_1=a$ when $a$ is finite, is covered by Theorem \ref{cor3}. The case
where $x_1>a$ is treated with the help of Lemma \ref{techn}. The reduction of the case of $N$
points to the case of a single point is made in a standard way via a unity partition, thanks to the
fact that Morrey space is a Banach function space, so that $|f(x)|\le |g(x)| \Longrightarrow
\|f\|_{p,\lb} \le \|g\|_{p,\lb}$.
\end{proof}
We arrive at the following result.

\begin{theorem}\label{final}
Let $-\infty<a<b<\infty$. The singular operator $S$ is bounded in the weighted Morrey space
$\mathcal{L}^{p,\lb}([a,b],\varrho), 1<p<\infty, 0\le \lb <1,$ with the weight
$$\varrho(x)=\prod_{k=1}^N\vi_k(|x-x_k|), \quad x_k \in [a,b]
$$
where $\vi_k\in \widetilde{W}\cap \left(V_{++}\cup V_{-+}\right)$, if
\begin{equation}\label{conda}
\vi_k\in \mathbb{Z}_{\frac{\lb}{p}+\frac{1}{p^\prime}}^\frac{\lb-1}{p}, \end{equation}
or
equivalently,
\begin{equation}\label{cond}
\frac{\lb-1}{p}< m(\vi_k) \le M(\vi_k)< \frac{\lb}{p}+\frac{1}{p^\prime}, \quad  k=1,2,...,N.
\end{equation}
\end{theorem}
\begin{proof}
The case of a single weight $\varrho(x)=\vi(x-a)$ follows from  Theorem \ref{theor} by the
pointwise  estimates of Corollary \ref{cor1}. The case of a single weight of the form
$\varrho(x)=\vi(|x-x_0|), x_0\in (a,b),$ is easily considered  with the aid of Lemma \ref{techn}.
The passage to the case of a product of such weights is done via the standard approaches.
\end{proof}

\begin{remark}\label{rem4}
When considering the case of non-power weights $\vi(x)$, for simplicity we supposed that  the
interval $[0,\ell]$ for the Hardy operators or the interval $[a,b]$ for the singular operator is
finite. The case of infinite interval also may be considered  for non-power weights, but then we
should somewhat modify definitions and introduce the Matusewska-Orlicz type indices responsible for
the behavior of weights not only at the origin but also at infinity.
\end{remark}

\section{On the non-weighted boundedness of the singular Cauchy operator along Carleson curves}
\label{Singular}
\setcounter{equation}{0}\setcounter{theorem}{0}

Our goal is to extend Theorem \ref{th}  to the case of the Cauchy singular operator \eqref{4ss}
along Carleson curves. We will obtain such an extension  from the boundedness of the maximal
operator in Morrey spaces (in a more general context of metric measure spaces), making use of the
pointwise estimate \eqref{alvar}.

\subsection{Maximal operator in Morrey spaces on metric measure spaces}
The boundedness of the maximal operator
$$Mf(x)=\sup\limits_{r>0}\intl_{B(x,r)} |f(y)|d\mu(y)$$
in a certain version of Morrey spaces  over metric measure spaces $(X,d,\mu)$  was proved in
\cite{26c}. In the recent paper \cite{KM}, the boundedness of the maximal operator on bounded
metric measure space was extended to the case of variable coefficients. The boundedness of the
operator $M$ in the space $L^{p,\lb}(X)$ under condition \eqref{2}, may be derived from \cite{26c}
and \cite{KM}. For completeness of the proof, we will present an independent and direct proof in
Theorem \ref{maximal}.

The following statement well known in the Euclidean setting (\cite{160ba}, Lemma 1;  \cite{644}, p.
53), for homogeneous spaces was proved in  \cite{491a}, Proposition 3.4.

\begin{lemma}\label{lemmaFS}
Let $X$ be a homogeneous metric measure space with $\mu(X)=\infty$.  Then the Feffermann-Stein
inequality
\begin{equation}\label{3}
 \intl_{X}(Mf)(y)^p\,w(y)\,d\mu(y) \le  \intl_{X}f(y)^p\, (Mw)(y)\,d\mu(y)
\end{equation}
 holds  for all non-negative functions $f,\,w$ on $X$.

\end{lemma}

\begin{theorem}\label{maximal}
Let $X$ be a metric measure space with $\mu(X)=\infty$. Under condition \eqref{2}, the maximal
operator $M$ is bounded in the space $L^{p,\lb}(X), 1<p<\infty, 0\le \lb<N$.
\end{theorem}
\begin{proof}
We follow the main lines of the proof in \cite{87b} for the case $X=\mathbb{R}^n$.  By
Fefferman-Stein inequality \eqref{3}, we obtain
$$
\intl_{B(x,r)}\left(Mf(y)\right)^{p}  dy\le
 C\intl_{X}|f(y)|^{p} M\chi_{B(x,r)}(y)\,
dy
$$
$$\le C\intl_{B(x,r)}|f(y)|^{p} (y)dy + C\sum\limits_{j=0}^\infty \intl_{B(x,2^{j+1}r)\backslash
B(x,2^{j}r)}|f(y)|^{p} M\chi_{B(x,r)}(y)\, dy. $$
We make use of the estimate
\begin{equation}\label{estim}
M\chi_{B(x,r)}(y)\le C\frac{r^N}{(d(x,y)+r)^N}, \quad x,y \in X,
\end{equation}
valid under condition \eqref{2}, which is well known in the Euclidean case and the proof in our
case is in main the same as, for instance, in \cite{69ab}, p. 160-161, thanks to condition
\eqref{2}. We then  obtain
$$
\intl_{B(x,r)}\left(Mf(y)\right)^{p}  dy\le C\intl_{B(x,r)}|f(y)|^{p} (y)dy +
C\sum\limits_{j=1}^\infty \intl_{B(x,2^{j+1}r)\backslash B(x,2^{j}r)}|f(y)|^{p} M\chi_{B(x,r)}(y)\,
dy$$
$$\le C\intl_{B(x,r)}|f(y)|^{p} (y)dy +
\sum\limits_{j=1}^\infty\frac{C}{[2^j+1]^N} \intl_{B(x,2^{j+1}r)}|f(y)|^{p}\, dy.$$
Hence
$$
\|Mf\|_{\mathcal{L}^{p,\lb}}=\sup\limits_{x,r}\frac{1}{r^\lb}\intl_{B(x,r)}\left(Mf(y)\right)^{p}
dy\le \sup\limits_{x,r}\frac{C}{r^\lb}\intl_{B(x,r)}|f(y)|^{p}  dy
$$
\begin{equation}\label{prosto}
+\sum\limits_{j=1}^\infty\frac{C}{[2^j+1]^{N-\lb}}\sup\limits_{x,r}
     \frac{1}{r^\lb}\intl_{B(x,r)}|f(y)|^{p}  dy = C_1
 \|f\|_{\mathcal{L}^{p,\lb}}.
\end{equation}
\end{proof}

Let
$$M^\#f(x)=\sup\limits_{r>0}\frac{1}{|B(x,r)|}\intl_{B(x,r)}\left|f(y)-f_{B(x,r)}\right|d\mu(y), \quad
f_{B(x,r)}= \intl_{B(x,r)}f(y)d\mu(y).$$

 To deal with the boundedness of the singular operator via pointwise estimate \eqref{alvar}, we
  also need the following Fefferman-Stein inequality
in Morrey-norms for metric measure spaces (proved in \cite{160bz} in the case $X=\mathbb{R}^n$).

\begin{lemma}\label{lemFSM}
Let $X$ be a metric measure space with $\mu(X)=\infty$. Under condition \eqref{2}
$$\|Mf\|_{L^{p,\lb}(X)}\le C \|M^\#f\|_{L^{p,\lb}(X)},\quad 1<p<\infty, 0\le \lb<N.$$
\end{lemma}
\begin{proof}
The proof is in fact the same as in  \cite{160bz}. We make use of the following weighted
Fefferman-Stein inequality in $L_p$-norms
\begin{equation}\label{weightedFS}
\intl_{X}|Mf(x)|^pw(x)dx \le C \intl_X |M^\#f(x)|^pw(x)dx, \quad w\in A_\infty, f\in L^p(X,w)
\end{equation}
valid for homogeneous metric measure spaces, see \cite{187}, p. 184. According to Coifman-Rochberg
\cite{96a} characterization of $A_1$, the function $\left[M\chi_{B(x,r)}\right]^\ve, \ 0<\ve<1$, is
in $A_1$ (see \cite{182za}, Proposition 3.1 for the case of homogeneous spaces). Since
$\chi_{B(x,r)}\le M\chi_{B(x,r)}\le \left[M\chi_{B(x,r)}\right]^\ve$, by \eqref{weightedFS} we
obtain
$$\intl_{B(x,r)}|Mf(y)|^pd\mu(y)\le C\intl_X |Mf(y)|^p \left[M\chi_{B(x,r)}(y)\right]^\ve
d\mu(y) \le C\intl_X |M^\#f(y)|^p \left[M\chi_{B(x,r)}(y)\right]^\ve d\mu(y)$$
$$\le \intl_{B(x,r)} |M^\#f(y)|^p  d\mu(y)+
\sum\limits_{j=0}^\infty  \frac{C}{(2^{j+1}+1)^{N\ve}}\intl_{B\left(x,2^{j+1}r\right)}
|M^\#f(y)|^pd\mu(y),$$ where \eqref{estim} have been used. Then similarly to estimations in
\eqref{prosto} we arrive at the statement of the lemma under the choice $\ve\in
\left(\frac{\lb}{n},1\right)$.
\end{proof}

\subsection{Singular Cauchy operator along Carleson curves in Morrey spaces; non-weighted case}

The following theorem was proved in a recent paper \cite{KM} in case of bounded curves, but in a
more general setting of variable exponents.

\begin{theorem}\label{theorsing}
Let $\Gm$ be a Carleson curve.  The singular operator $S_\Gm$ is bounded in the space
$L^{p,\lb}(\Gm), \quad 1<p<\infty, 0\le \lb<1$.
\end{theorem}
\begin{proof} Since a function on a bounded Carleson curve may be continued by zero to an infinite
Carleson curve with the preservation of the Morrey space, it suffices to consider the case where
$\Gm$ is an infinite curve.

Having the pointwise estimate  \eqref{alvar} in mind, we make use of the property of the norm
$$\|f\|_{p,\lb}=\|f^s\|_{\frac{p}{s},\lb}, \ \ 0<s<1$$
and have
$$\|S_\Gm f\|_{p,\lb}=\|(S_\Gm f)^s\|_{\frac{p}{s},\lb}\le \|M[(S_\Gm f)^s]\|_{\frac{p}{s},\lb}.$$
Then by  Lemma \ref{lemFSM} and estimate \eqref{alvar} we obtain
$$\|S_\Gm f\|_{p,\lb}\le C \|M^\#[(S_\Gm f)^s]\|_{\frac{p}{s},\lb}\le
C \|(Mf)^s]\|_{\frac{p}{s},\lb}=C \|Mf\|_{p,\lb}.$$ It remains to apply Theorem \ref{maximal}.
\end{proof}

\section{Singular Cauchy operator along Carleson curves in weighted Morrey spaces}\label{finalsection}
\setcounter{equation}{0}\setcounter{theorem}{0}

 Let $\Gm$ be a Carleson curve, $t_k\in\Gm,
k=1,...,N,$ and $\varrho$ a weight of form \eqref{vydel} with $ \vi_k\in \widetilde{W}\cap
\left(V_{++}\cup V_{-+}\right)$.
\begin{theorem}\label{final}\\
I) \ Let the curve $\Gm$ satisfy the arc-chord condition. The operator $S_\Gm$ is bounded in the
Morrey space  $\mathcal{L}^{p,\lb}, \ 1<p<\infty, 0\le \lb<1$, with weight \eqref{vydel}, if
condition \eqref{conda} (or equivalent condition \eqref{cond}) is satisfied. \\
II) \ Let the curve $\Gm$ satisfy the arc-chord condition and be
smooth in  neighborhoods of the nodes $t_k, k=1,...,N,$ of  the
weight. In the case of power weights $\vi_k(r)=r^{\al_k}$ the
corresponding condition \eqref{cond}, that is, $\frac{\lb-1}{p}<
\al_k < \frac{\lb}{p}+\frac{1}{p^\prime}, \quad k=1,2,...,N, $ is
also necessary for the
boundedness.\\
III) \ Statement I)  remains valid on a Carleson curve  $\Gm$ for the space $L_\ast^{p,\lb}(\Gm)$,
under the assumption that the curve $\Gm$
 has the arc-chord property only  at the nodes $t_k, k=1,...,N$ of  the weight.
\end{theorem}

\begin{proof}

I) \ As usual, we may consider only the case of a single weight $\varrho(t)=\vi(|t-t_0|), t_0\in\Gm
.$ In view of Theorem \ref{theorsing}, it suffices to prove the boundedness of the operator
\begin{equation}\label{replace7b}
K f(t): \ = \ \left(\varrho S_\Gm\frac{1}{\varrho}-S_\Gm \right) f(t) = \intl_\Gm
K(t,\tau)f(\tau)\, d\mu(\tau),
\end{equation}
 where
$K(t,\tau): = \frac{\varrho(t)-\varrho(\tau)}{\varrho(\tau)(\tau-t)}=
\frac{\vi(|t-t_0|)-\vi(|\tau-t_0|)}{\vi(|\tau-t_0|)(\tau-t)}.$ By the definition of the classes
$\mathbf{V}_{++}, \mathbf{V}_{-+}$, we observe that the kernel $K(t,\tau)$ admits the estimate
\begin{equation}\label{replace8}
|K(t,\tau)|\le \left\{\begin{array}{ll}\cfrac{C\vi(|t-t_0|)}{|t-t_0|\vi(|\tau-t_0|)},  &\
\textrm{if} \ \ |\tau-t_0|<|t-t_0|\\
\cfrac{C}{|t-t_0|}, &\  \textrm{if} \ \ |\tau-t_0|>|t-t_0|
\end{array} \right.
\end{equation}
when $\vi\in \mathbf{V}_{++}$, and
\begin{equation}\label{replace9}
|K(t,\tau)|\le \left\{\begin{array}{ll}\cfrac{C}{|t-t_0|},  &\  \textrm{if} \ \ |\tau-t_0|<|t-t_0|\\
\cfrac{C\vi(|t-t_0|)}{|\tau-t_0|\vi(|t-t_0|)}, &\  \textrm{if} \ \ |\tau-t_0|>|t-t_0|
\end{array} \right.
\end{equation}
when $\vi\in \mathbf{V}_{-+}$.
 Then the operator $K$ is dominated by the weighted
Hardy type operators
\begin{equation}\label{replace10}
|K f(t)| \le C \frac{\vi(|t-t_0|)}{|t-t_0|}\intl_{\Gm_t}
\frac{|f(\tau)|d\mu(\tau)}{\vi(|\tau-t_0|)} + C \intl_{\Gm\backslash\Gm_t}
\frac{|f(\tau)|d\mu(\tau)}{|\tau-t_0|}
\end{equation}
when $\vi \in \mathbb{V}_{++}$, and
\begin{equation}\label{replace10ax}
|K f(t)| \le \frac{C}{|t-t_0|}\intl_{\Gm_t} |f(\tau)|d\mu(\tau) + C
\vi(|t-t_0|)\intl_{\Gm\backslash\Gm_t} \frac{|f(\tau)|d\mu(\tau)}{|\tau-t_0|\vi(|\tau-t_0|)}
\end{equation}
when $\vi \in \mathbb{V}_{-+}$, where  $\Gm_t=\{\tau\in\Gm: |\tau-t_0|<|t-t_0|\}$.

 Note that the condition $\vi\in W_1\cap \widetilde{W}$ guarantees
the equivalence
\begin{equation}\label{replacejdhg}
C_1 \vi(|s-s_0|) \le \vi(|t-t_0|) \le C_2 \vi(|s-s_0|), \ \ t=t(s), \ t_0=t(s_0)
\end{equation}
on curves satisfying the arc-chord condition at the point $t_0$. Since
$\frac{\lb}{p}+\frac{1}{p^\prime}<1$,  condition  \eqref{cond} implies that $\vi\in W_1$ and
therefore, equivalence \eqref{replacejdhg} holds  under the conditions of the theorem.

Without a loss of generality we may assume that the the arc length counts from the point $t_0$,
that is, $s_0=0$ (which may always be supposed in the case of a  closed curve, while for an open
curve this means that $t_0$ must be an end-point; the case where $t_0$ is not, may be easily
covered similarly to Lemma \ref{techn}). Then, in view of \eqref{replacejdhg}, it is easily seen
that estimates \eqref{replace10} and \eqref{replace10ax} are equivalent to the following
"arc-length" form
\begin{equation}\label{replace10a}
|K f(t)| \le C \frac{\vi(s)}{s}\intl_0^s \frac{|f_\ast(\sg)|d\sg}{\vi(\sg)} + C \intl_s^\ell
\frac{|f_\ast(\sg)|}{\sg}d\sg, \quad t=t(s),
\end{equation}
when $\vi \in \mathbb{V}_{++}$, and
\begin{equation}\label{replace10axa}
|K f(t)| \le \frac{C}{s}\intl_0^s |f_\ast(\sg)|d\sg + C \vi(s)\intl_s^\ell
\frac{|f_\ast(\sg)|d\sg}{\sg\vi(\sg)}, \quad t=t(s),
\end{equation}
when $\vi \in \mathbb{V}_{-+}$
 (taking into account that $s_0=0$), where $f_\ast(\sg)=f[t(s)]$.
 It remains to apply Theorem \ref{theor} to the Hardy operators on the right-hand side of
 \eqref{replace10a}-\eqref{replace10axa} keeping  Remark \ref{rem6} in mind.

II) \ The proof of  the necessity of conditions $\frac{\lb-1}{p}<
\al_k < \frac{\lb}{p}+\frac{1}{p^\prime}, \quad k=1,2,...,N, $  in
the case of power weights follows the same line as
 in the
proof of \textit{the "only if"} part of Theorem \ref{cor3}, with
corresponding modifications. We explain the necessary modification
for \eqref{asympt}. Now we have
$$S_\al f (t)= \frac{|t-t_0|^\al}{\pi}\intl_\Gm \frac{|\tau-t_0|^{-\al}f(\tau)}{\tau -t}d\tau
$$ $$ = \frac{|t-t_0|^\al}{\pi}\intl_0^\ell
\frac{|t(\sg)-t(s_0)|^{-\al}f(\tau)\tau^\prime(\sg)}{t(\sg)
-t(s)}d\sg, \quad t_0=t(s_0)\in \Gm.$$ We choose
$$f(\tau)=f[t(\sg)]=(\sg-s_0)^{\frac{\lb-1}{p}-\al}_+ \ \cdot
\frac{t(\sg)-t(s_0)}{\sg-s_0} \cdot
\frac{|t(\sg)-t(s_0)|^\al}{t^\prime(\sg)}$$ where
$(\sg-s_0)^{\frac{\lb-1}{p}-\al}_+=\left\{\begin{array}{ccc}(\sg-s_0)^{\frac{\lb-1}{p}-\al}
 & ,& \sg>s_0\\
0 & , & \sg<s_0\end{array} \right.$ and it is assumed that $s_0\ne
\ell,$ the arguments being easily modified for the case
$s_0=\ell.$ By the smoothness of the curve near the point $t_0$,
that is, the continuity of $t^\prime(\sg)$ near $\sg=s_0$ and the
condition $|t^\prime(\sg)|\equiv 1$, we see that
$$|f(\tau)|\le C |\tau-t_0|^\frac{\lb-1}{p}\in \mathcal{L}^{p,\lb}(\Gm)$$
 by  Lemma  \ref{replacelem0}.
However, under this choice of $f(\tau)$, by the continuity of
$t^\prime(\sg)$ it is easy to see that
$$S_\al f(t) \sim c|t-t_0|^\al \quad \textrm{with} \quad c\ne 0
 \quad \textrm{as} \quad t\to t_0,$$
 as in \eqref{asympt}.

 Finally, it remains to observe that property
\eqref{logarifm} of the singular integral is known to be valid on
an arbitrary Carleson curve, see for instance \cite{63}, pages
118-120.

\end{proof}

\begin{remark}\label{rem5}
In case one uses  weights of the form
$\prod_{k=1}^N\vi_k(|s-s_k|), \ 0\le s_1<s_2<\cdots <s_N < \ell$,
the requirement for $\Gm$ to satisfy the arc-chord condition in
Part \textit{III)} of Theorem \ref{final} may be  omitted as is
easily seen from the proof.
\end{remark}

\section*{Acknowledgements}

 This work was made   under the project
``Variable Exponent Analysis" supported by INTAS grant Nr. 06-1000017-8792.

%\bibliographystyle{plain}

%\bibliography{N-Morrey}

\end{document}